\documentclass[11pt,a4paper]{amsart}
\usepackage[T1]{fontenc}
\usepackage{lmodern,microtype}
\usepackage[margin=29mm]{geometry}
\usepackage{amsmath,amssymb,amsthm,mathtools,amscd}
\usepackage[colorlinks=true,linkcolor=black,citecolor=blue,urlcolor=blue]{hyperref}
\hypersetup{pdftitle={Morita bundle gerbes},pdfauthor={Andrei V. Ershov}}
\newtheorem{theorem}{Theorem}[section]
\newtheorem{proposition}[theorem]{Proposition}

\theoremstyle{definition}
\newtheorem{definition}[theorem]{Definition}
\newtheorem{example}[theorem]{Example}
\theoremstyle{remark}

\newcommand{\C}{\mathbb C}\newcommand{\Z}{\mathbb Z}
\newcommand{\CX}{\mathcal O^\times}
\newcommand{\one}{\mathbf 1}\newcommand{\HH}{\mathcal H}
\newcommand{\U}{\mathrm U}\newcommand{\PU}{\mathrm{PU}}
\newcommand{\KU}{\mathrm{KU}}\newcommand{\GL}{\mathrm{GL}}
\newcommand{\DD}{\mathrm{DD}}\newcommand{\Br}{\mathrm{Br}}
\DeclareMathOperator{\End}{End}\DeclareMathOperator{\Hom}{Hom}
\DeclareMathOperator{\Aut}{Aut}\DeclareMathOperator{\Pic}{Pic}
\DeclareMathOperator{\Mod}{Mod}\DeclareMathOperator{\id}{id}

\newcommand{\arxiv}[1]{\href{https://arxiv.org/abs/#1}{arXiv:#1}}
\begin{document}
\title{Morita bundle gerbes}

\author{Andrei V. Ershov}
\thanks{Partially supported by the RFBR (grant 14-01-00007-a
``Analytic methods in noncommutative geometry and topology'').}
\address{Moscow Institute of Physics and Technology, Dolgoprudny, Russia}
\email{ershov.andrei@gmail.com}
\date{Revised manuscript, November 2026}
\begin{abstract}
The aim of this paper is to give a survey of the theory of bundle gerbes.
We describe how local Morita equivalences replace a cocycle of matrix
algebras and bimodules by a cocycle of complex lines. A positive-rank
gerbe module identifies its module category with modules over an Azumaya
bundle. We determine the Brauer classes represented by matrix degrees
dividing powers of a prescribed integer, and compare Morita equivalence
with unital UHF stabilization on the two-sphere. For Morita 2-bundle
gerbes we carry out the reduction through the associator and the
equivalences of cocycles. This identifies their equivalence classes on
good-cover bases with fourth integral cohomology.
\end{abstract}
\maketitle

\section{Introduction}

A bundle gerbe records an obstruction to gluing local objects. Its
Dixmier--Douady class determines whether the obstruction vanishes, but a
trivialization also supplies the gluing maps. Two trivializations can
differ by a line bundle, and tensoring by that line can act nontrivially
on the resulting \(K\)-group. This is one reason to keep the morphisms of
twists in a geometric model of twisted \(K\)-theory. One instance is
topological \(T\)-duality, where a specified identification of twists
enters the construction of an isomorphism of twisted \(K\)-groups
\cite[\S3.1]{BSch}.

In our approach we emphasize the role of Morita equivalence, which
plays the unifying role in this theory. On a double overlap one can
glue two matrix algebras by an invertible bimodule. Choosing a simple
module on each patch turns the bimodule into a line. Conversely, a
finite-rank module for the resulting gerbe recovers a global matrix
algebra bundle. Section~\ref{sec:ordinary} proves the first comparison.
Section~\ref{sec:modules} gives the second one in the form of Karoubi's
equivalence of module categories and keeps track of its variance.

The ordinary classification and the connection with twisted \(K\)-theory
are due to earlier work, in particular
\cite{Murray,MS,BCMMS,Karoubi2}. Pennig constructed Morita bundle gerbes
for \(C^*\)-algebras \cite{Pennig}. Kristel, Ludewig and Waldorf
\cite[\S3.1]{KLW} compare the finite-dimensional model of the earlier
version of this paper \cite{original} with ungraded complex 2-line
bundles of Morita class \(\C\) on smooth manifolds. Here the arrows
used for classifying twists are the invertible ones. Modules, including
those of higher rank, form a separate category.

Two questions illustrate what can be seen at a finite matrix size.
Which torsion classes admit a representative of degree \(k^m\) for some
\(m\)? Which information about a given representative survives unital
stabilization? Sections~\ref{sec:degree} and~\ref{sec:uhf} answer the
first question and give a two-sphere example for the second. In
Section~\ref{sec:higher} the same change of local presentation takes
place one degree higher. The associativity isomorphism becomes a scalar
on quadruple overlaps. We compute its change under a Morita equivalence
and deduce a classification of equivalence classes.

\subsection*{The place of the degree-three twist}
For a spectrum \(E\), \(E^0(X)=[X,\Omega^\infty E]\). A twist by
\(KU\)-module lines is classified by a map to \(B\GL_1(\KU)\), where
\(\GL_1(\KU)\) is the space of module automorphisms of the free
rank-one \(\KU\)-module. Twisted cohomology is obtained from the
spectrum of sections of the associated parametrized module
\cite{ABGHR,MayS}. In the usual notation one has a decomposition of spaces
\[
 B\GL_1(\KU)\simeq K(\Z/2,1)\times K(\Z,3)\times BBSU_\otimes .
\]
The connected splitting \(BU_\otimes\simeq K(\Z,2)\times BSU_\otimes\)
is a splitting of infinite loop spaces. The tensor product on the
low-degree graded coordinates \(w\in H^1(X;\Z/2)\),
\(h\in H^3(X;\Z)\) is
\[
 (w,h)(w',h')=(w+w',\,h+h'+\beta(w\smile w')),
\]
where \(\beta\) is the integral Bockstein
\cite[Prop.~3.2 and Eq.~(24)]{DMP}. For example, for the two
degree-one generators on \(\mathbb{RP}^2\times\mathbb{RP}^2\),
the reduction of \(\beta(ww')\) is
\(\mathrm{Sq}^1(ww')=w^2w'+ww'^2\ne0\). The ungraded gerbes in
this paper use \(w=0\), where the tensor law on \(h\) is addition.

{\noindent \bf Acknowledgments.\;}
I would like to express my deepest gratitude to Professor Doctor Thomas Schick
for numerous inspirational discussions and valuable contributions to some parts of this text.

\smallskip\noindent\textit{Use of AI.}
OpenAI's ChatGPT assisted with reviewing the earlier manuscript,
checking calculations and references, and drafting this revision.

\subsection*{Conventions}
Our good-cover arguments apply to finite polyhedra and to paracompact,
second-countable smooth manifolds, with continuous bundles in both cases.
Write \(U_{i_0\ldots i_p}=\bigcap_a U_{i_a}\). All nonempty finite
intersections in a good cover are contractible. Data on different covers
are compared on a common good refinement. Scalars belong to the sheaf
\(\CX(U)=C(U,\C^\times)\); all scalar intertwiners are complex linear.
For cochains we use all index tuples and the alternating multiplicative
differential \((\delta f)_{i_0\ldots i_{p+1}}
=\prod_a f_{i_0\ldots\widehat i_a\ldots i_{p+1}}^{(-1)^a}\).
The exponential sequence of sheaves
\[
 0\longrightarrow\underline{\Z}\longrightarrow
 C(-,\C)\xrightarrow{\exp(2\pi i\,\cdot)}\CX\longrightarrow 1
\]
identifies \(\check H^p(X;\CX)\) with \(H^{p+1}(X;\Z)\) for \(p\ge1\).
Indeed \(C(-,\C)\) is fine, and each good intersection is acyclic for
\(\CX\) in positive degrees by the same sequence. Consequently the
cochains of any one good cover compute these groups. The integral
cohomology of the spaces just specified agrees with singular cohomology.
Compact Hausdorff spaces require more care. The Cantor set has no
nonempty contractible open subset. Section~\ref{sec:degree} states
separately the finite Brauer theorem valid in that generality.

\section{Bundle gerbes and Morita descent}\label{sec:ordinary}

Bundle gerbes over a base space \(X\) form a weak monoidal 2-groupoid.
It is a categorification of the group \(H^3(X;\Z)\) in the sense
that the equivalence classes of objects form that group under tensor
product. We recall this description in a form that also applies to the
matrix presentation.

\subsection{The line-bundle presentation}
A bundle gerbe \((L,\theta,\mathcal U)\) consists of lines \(L_{ij}\)
on \(U_{ij}\) and isomorphisms
\(\theta_{ijk}:L_{ij}\otimes L_{jk}\to L_{ik}\) on \(U_{ijk}\).
Associativity requires the square
\[
\begin{CD}
 L_{ij}\otimes L_{jk}\otimes L_{kt} @>{1\otimes\theta_{jkt}}>>
 L_{ij}\otimes L_{jt}\\
 @V{\theta_{ijk}\otimes 1}VV @VV{\theta_{ijt}}V\\
 L_{ik}\otimes L_{kt} @>{\theta_{ikt}}>> L_{it}
\end{CD}
\]
to commute. Multiplication on \(L_{ii}\) has a canonical unit. In a
local frame \(s\) with \(s^2=as\), it is \(s/a\); this expression is
independent of the frame. Associativity implies that it acts as the
identity on \(L_{ij}\). Thus \(L_{ii}\) is canonically trivial and
the product identifies \(L_{ji}\) with \(L_{ij}^{\vee}\).

An equivalence \(L\to L'\) consists of lines \(T_i\) and isomorphisms
\[
 \phi_{ij}: L_{ij}\otimes T_j\longrightarrow T_i\otimes L'_{ij}.
\]
On triples, applying first \(\phi_{jk}\), then \(\phi_{ij}\), and
then \(\theta'_{ijk}\) agrees with applying \(\theta_{ijk}\) and
\(\phi_{ik}\). Between \((T,\phi)\) and \((T',\phi')\) a 2-arrow
is a family \(d_i:T_i\to T'_i\) of isomorphisms with
\((d_i\otimes1)\phi_{ij}=\phi'_{ij}(1\otimes d_j)\).
Composition tensors the patch lines. The dual lines give an inverse
equivalence, using \(\phi_{ij}^{-1}\) and evaluation. In particular,
if \(T:L\to L'\) and \(T^\vee:L'\to L\), the composites are
\(T^\vee T\simeq 1_L\) and \(TT^\vee\simeq1_{L'}\).
We obtain a bicategory with invertible arrows up to isomorphism.
Tensoring gerbes is defined term by term, using the symmetry of lines.
Its unit is the constant line with its usual multiplication; dualizing
gives an inverse on equivalence classes.

In terms of the \v{C}ech groupoid, a patch point is sent to the
unique object of the Picard bicategory of complex lines, an overlap
arrow to a line, and a composable pair to its multiplication
isomorphism. The square above is precisely the condition on three
composable arrows. The lines here vary continuously over the space of
arrows. The same description with matrix algebras as objects explains
the appearance of balanced tensor products below.

\begin{proposition}\label{p:dd}
The Dixmier--Douady class gives a natural group isomorphism between
stable equivalence classes of line-bundle gerbes and \(H^3(X;\Z)\).
\end{proposition}
\begin{proof}
Choose frames \(s_{ij}\) on a good cover and put
\(\theta(s_{ij}\otimes s_{jk})=\lambda_{ijk}s_{ik}\).
The square above says \(\delta\lambda=1\).
A change \(s_{ij}\mapsto c_{ij}s_{ij}\) changes \(\lambda\) by
\(\delta c\). In patch frames \(u_i\) for an equivalence write
\(\phi_{ij}(s_{ij}\otimes u_j)=\mu_{ij}u_i\otimes s'_{ij}\).
Its triangle condition is precisely
\begin{equation}\label{e:scalar-arrow}
 \mu_{ij}\mu_{jk}\lambda'_{ijk}=\lambda_{ijk}\mu_{ik}.
\end{equation}
Conversely a 2-cocycle is realized by trivial overlap lines with
multiplication \(\lambda\); a 1-cochain satisfying
\eqref{e:scalar-arrow} gives an equivalence between these presentations.
The exponential connecting isomorphism takes \([\lambda]\) to
\(\DD(L)\). Tensor products multiply the cocycles, and pullback
restricts the construction after refining the inverse-image cover.
This also proves independence of the chosen good cover.
\end{proof}

A right trivialization is an equivalence \(L\to\one\).
Its patch lines satisfy \(L_{ij}\otimes T_j\simeq T_i\).
The set of its isomorphism classes, when nonempty, is a torsor
for \(\Pic(X)\): the ratios \(T'_i\otimes T_i^\vee\) have ordinary
descent isomorphisms and glue to a line on \(X\). Conversely tensor
by any such line produces a trivialization. The same argument for
endomorphisms identifies the group of self-equivalences up to 2-arrow
with \(\Pic(X)\). The 2-automorphisms of a fixed equivalence are
\(C(X,\C^\times)\).

\subsection{Matrix labels and bimodules}
We use the convention that an arrow \(P:A\to B\) is a left \(B\),
right \(A\) bimodule. Composition \(QP\) is \(Q\otimes_B P\).
We suppress the canonical balancing and reassociation isomorphisms.
A Morita equivalence has an inverse \(P^{-1}\) with evaluation and
coevaluation satisfying the triangle identities. For
\(A=\End(V)\) one may take \(V:\C\to A\) and
\(V^\vee:A\to\C\). A scalar bimodule is a complex vector space on
which the two scalar actions agree; it is Morita-invertible exactly
when it is a line.

\begin{definition}\label{d:ordinary}
A Morita bundle gerbe has matrix algebra bundles \(A_i\) on \(U_i\),
Morita equivalences \(M_{ji}:A_i\to A_j\) over \(U_{ij}\), and
isomorphisms
\[
 m_{ijk}: M_{kj}M_{ji}\longrightarrow M_{ki}.
\]
The two contractions of \(M_{tk}M_{kj}M_{ji}\) coincide. An
equivalence to \((A',M',m')\) has Morita-invertible
\((A'_i,A_i)\)-bimodules \(N_i\) and isomorphisms
\[
 u_{ij}:M'_{ji}N_i\longrightarrow N_jM_{ji}.
\]
On \(M'_{kj}M'_{ji}N_i\), contraction by \(m'\) followed by
\(u_{ik}\) agrees with \(u_{ij}\), then \(u_{jk}\), then \(m\),
with the indicated identity factors. A 2-arrow \(d:N\to N'\)
satisfies \((d_j\,1)u_{ij}=u'_{ij}(1\,d_i)\).
All intertwiners specifying equivalences or 2-arrows are isomorphisms.
\end{definition}

For scalar labels identify \(M_{ji}\) with \(L_{ij}\), using the
symmetry of lines to order the factors of \(m\). Under this convention
\(u_{ij}^{-1}\), with the same symmetry, is the map \(\phi_{ij}\)
above. The invertibility condition on \(N_i\) is essential. Over a
point, the identity compatibility would otherwise allow the bimodule
\(\C^2\) between two scalar objects. Its tensor product with any vector
space has dimension different from one.

\begin{theorem}\label{t:ordinary}
On a fixed good cover, the inclusion of line-bundle gerbes into Morita
bundle gerbes with the arrows of Definition~\ref{d:ordinary} is a
biequivalence. After refinement their groups of equivalence classes
are naturally identified by \(\DD\).
\end{theorem}
\begin{proof}
Choose splittings \(S_i:A_i\to\C\) on the contractible patches.
The bimodule
\[
 L_{ij}=S_j M_{ji} S_i^{-1}
\]
is a line. Contract \(S_j^{-1}S_j\) in \(L_{jk}L_{ij}\) and use
\(m_{ijk}\) to produce multiplication. On fourfold overlaps
contractions at different patch indices commute, and the two remaining
maps agree by associativity of \(m\). The maps
\(L_{ij}S_i\cong S_jM_{ji}\) give the equivalence to the scalar
presentation. Between two scalar objects, an invertible patch bimodule
is precisely a line and its isomorphisms are exactly the 2-arrows
already specified. Hence the inclusion induces equivalences of hom
groupoids and is essentially surjective. Proposition~\ref{p:dd}
gives the last assertion.
\end{proof}

A global algebra bundle \(A\) gives a gerbe with patch restrictions of
\(A\), regular overlap bimodules and multiplication. With local simple
left modules \(V_i\), the resulting line on \(U_{ij}\) can be written
\(\Hom_A(V_j,V_i)\). Composition gives its product. Thus if
\(g_{ij}\) is a unitary lift of a projective transition from chart
\(j\) to chart \(i\), its scalar cocycle is read from
\(g_{ij}g_{jk}=\lambda_{ijk}g_{ik}\). We use this sign for the
\emph{lifting} class of a projective bundle throughout.

\subsection{The algebraic analogue and the classifying space}
For a commutative ring \(R\), the Brauer--Picard bicategory has
Azumaya \(R\)-algebras, invertible \(R\)-central bimodules, and
bimodule isomorphisms. At the three successive levels one recovers
\(\Br(R)\), \(\Pic(R)\), and \(R^\times\). For \(R=C(X)\)
the last group is \(C(X,\C^\times)\). With ordinary algebraic
intertwiners, pullback along \([0,1]\to *\) misses, for example,
\(t\mapsto e^t\) among the automorphisms of the identity bimodule.
Homotopy invariance of \(\DD\) concerns equivalence classes.

Let \(\HH\) be a separable infinite-dimensional complex Hilbert space,
and use the norm topology. The circle extension
\(\U(1)\to\U(\HH)\to\PU(\HH)\) gives a line on \(\PU(\HH)\);
its unit vectors are the unitary lifts of a projective operator, and they multiply by composition. Pulling this line back along the
transition maps \(p_{ij}\) of a principal \(\PU(\HH)\)-bundle gives
its lifting gerbe. The cocycle of this gerbe is
\(g_{ij}g_{jk}g_{ik}^{-1}\) for chosen unitary lifts.
By Kuiper's theorem \cite{Kuiper}, \(\U(\HH)\) is contractible,
so \(\PU(\HH)\simeq K(\Z,2)\). The delooped central extension
has homotopy fiber \(B\U(\HH)\) over
\(B\PU(\HH)\to B^2\U(1)\), hence this last map is a weak
equivalence. Its obstruction cocycle is precisely the one just
computed. Consequently \(B\PU(\HH)\simeq K(\Z,3)\) represents the
classification in Proposition~\ref{p:dd}. Under suspension the same
construction takes a clutching map \(f:Y\to\PU(\HH)\) to the
suspension of the Chern class of the pulled-back circle bundle, up
to the orientation of the suspension. Indeed the line on the equator
is \(f^*\U(\HH)\times_{\U(1)}\C\), and the connecting homomorphism
for the two cones is the suspension isomorphism.

\section{Bundle gerbe modules and Azumaya bundles}\label{sec:modules}

\begin{definition}
A finite-rank right module for \(L\) consists of vector bundles
\(E_i\to U_i\) and isomorphisms
\(\epsilon_{ij}:E_i\otimes L_{ij}\to E_j\) such that
\[
 \epsilon_{jk}(\epsilon_{ij}\otimes1)
 =\epsilon_{ik}(1\otimes\theta_{ijk})
 \quad\text{on } E_i\otimes L_{ij}\otimes L_{jk}.
\]
Module maps are compatible families of linear bundle maps. Write
\(\Mod(L)\) for this additive category. A left module has the
oppositely directed action \(L_{ij}\otimes F_j\to F_i\).
\end{definition}

In particular, \(\End(E_i)\to\End(E_j)\) sends \(a\) to
\(\epsilon_{ij}(a\otimes1)\epsilon_{ij}^{-1}\). The module identity
implies ordinary descent, so these are the local charts of a global
algebra bundle \(A_E\). When \(E\) has positive rank on each component
this is an Azumaya bundle. Taking determinants gives
\(\det E_i\otimes L_{ij}^{\otimes r}\cong\det E_j\) on a component
of rank \(r\). Thus \(r\DD(L)=0\) there.

\begin{theorem}[Module comparison {\cite[Thm.~3.5]{Karoubi2}}]
\label{t:modules}
Let \(E\in\Mod(L)\) have positive finite rank on every component.
There is an additive equivalence
\[
 \Mod(L)\longrightarrow\Mod_{\mathrm{right}}(A_E),\qquad
 F\longmapsto\{\Hom(E_i,F_i)\}_{\mathrm{desc}},
\]
where the target consists of finite projective right module bundles.
Its inverse sends \(H\) to the family
\(H|_{U_i}\otimes_{A_E|_{U_i}} E_i\).
\end{theorem}
\begin{proof}
On \(U_{ij}\) a homomorphism \(v:E_i\to F_i\) is transported to
\(\epsilon^F_{ij}(v\otimes1)(\epsilon^E_{ij})^{-1}\).
These transports respect composition with local endomorphisms of \(E\);
the right action is precomposition. On triples the two occurrences of
\(\theta_{ijk}\) cancel. Thus the local Hom bundles glue. In the
inverse construction \(1\otimes\epsilon^E_{ij}\) gives the gerbe
action on \(H_i\otimes_{A_i}E_i\), with \(A_i=\End(E_i)\).
Both composites have their usual Morita identifications. Explicitly,
\[
 \Hom(E_i,F_i)\otimes_{A_i}E_i\to F_i,\quad v\otimes e\mapsto v(e),
 \qquad H_i\to\Hom(E_i,H_i\otimes_{A_i}E_i),\quad
 h\mapsto(e\mapsto h\otimes e).
\]
They are fiberwise isomorphisms by matrix Morita equivalence and commute
with the displayed overlap maps. Positive rank makes \(E_i\) a Morita
generator on every fiber.
\end{proof}

\paragraph{The sign convention.}
Write \(\epsilon_{ij}(-\otimes s_{ij})=t_{ij}:E_i\to E_j\).
Then \(t_{jk}t_{ij}=\lambda_{ijk}t_{ik}\). In local vector-space
coordinates the projective transition of \(A_E\) from chart \(j\)
to chart \(i\) has lift \(g_{ij}=t_{ij}^{-1}\); invertible complex
lifts give the same obstruction as unitary lifts after reduction of
structure group. Hence
\[
 g_{ij}g_{jk}=\lambda_{ijk}^{-1}g_{ik},\qquad
 \DD(A_E)=-\DD(L),\qquad \DD(A_E^{\mathrm{op}})=\DD(L).
\]
The sign is consistent with the right action in
Theorem~\ref{t:modules}. A right \(A_E\)-module is a left
\(A_E^{\mathrm{op}}\)-module. Equivalently, one can tensor a right
twisted module with a left twisted module: the two actions cancel.
Tensoring two right modules produces a module for the tensor product
of their gerbes. The untwisted Hom construction cancels two copies
of the same twist.

\paragraph{Equivalences, operations, and trivializations.}
An equivalence \((T,\phi):L\to L'\) carries \(E_i\) to
\(E_i\otimes T_i\). Its \(L'\)-action first uses \(\phi_{ij}^{-1}\)
on \(T_i\otimes L'_{ij}\), then the \(L\)-action on \(E_i\).
Its inverse uses \(T^\vee\). The induced isomorphism on Grothendieck
groups depends on the chosen equivalence. For \(L=\one\) it can be
tensoring by any line on \(X\). For example on \(S^2\), the Hopf
line sends \(1\in K^0(S^2)\) to its own class, which has nonzero first
Chern class. A chosen trivialization therefore gives a specific
isomorphism \(K^0(L)\cong K^0(X)\).

A choice of trivialization gives the equivalence with ordinary vector
bundles directly. If \(T:L\to\one\), the patches \(E_i\otimes T_i\)
glue, and the inverse sends a global bundle \(F\) to
\(F|_{U_i}\otimes T_i^\vee\). In particular two trivializations that
differ by a global line \(J\) give descended bundles differing by
tensor product with \(J\).

For torsion twists on finite CW bases, write \(K^0(L)\) for the
Grothendieck group of \(\Mod(L)\). Taking direct sums, tensoring by
ordinary vector bundles and pulling back are compatible with the
actions. They give a \(K^0(X)\)-module structure, natural pullbacks
and pairings \(K^0(L)\otimes K^0(L')\to K^0(L\otimes L')\).
For a positive-rank generator, Theorem~\ref{t:modules} identifies
this group with the \(K_0\) of the section algebra \(\Gamma(X,A_E)\)
on compact \(X\). The relation with topological twisted
\(K\)-theory is developed in \cite{BCMMS,Karoubi2}.
For a Morita bundle gerbe in Definition~\ref{d:ordinary}, the
corresponding left-module description takes left \(A_i\)-modules
\(V_i\) and isomorphisms \(M_{ji}V_i\to V_j\) compatible with
\(m_{ijk}\). For scalar labels this gives the action
\(L_{ij}\otimes V_i\to V_j\), the right-module convention above
after the symmetry of lines.

\section{Finite degree and the Brauer group}\label{sec:degree}

Let \(\Br_{\mathrm{fin}}(X)\) denote the group of Morita classes of
finite-dimensional Azumaya bundles. The lifting map for a single
degree is a pointed-set map
\[
 d_n:[X,B\PU(n)]\longrightarrow H^3(X;\Z).
\]
Its source has no prescribed tensor-product group law at fixed degree:
tensoring multiplies degrees. Write \(d_n(P)\) for the class of the
lifting gerbe. The determinant of its lifts shows \(nd_n(P)=0\).
For a finite CW complex the topological Brauer theorem identifies
\(\Br_{\mathrm{fin}}(X)\) with \(H^3(X;\Z)_{\mathrm{tors}}\).
Hornbostel and Schr\"oer \cite[Thm.~1.6]{HS} prove the compact
Hausdorff version with integral sheaf cohomology. By contrast,
ordinary gerbe classes range over all \(H^3\). The generator on \(S^3\)
is an immediate example outside \(\Br_{\mathrm{fin}}(S^3)\).

On our compact good-cover bases a gerbe admits a positive finite-rank
module precisely when its class is torsion. Necessity follows by
determinants. For sufficiency choose a projective bundle of lifting
class \(-\DD(L)\), using the Brauer theorem. On a good cover choose
lifts \(g_{ij}\) of its transitions. Their inverse scalar defect
represents the class of \(\lambda\) and can be made equal to
\(\lambda\) by a scalar change of lifts. The maps \(t_{ij}=g_{ij}^{-1}\)
now obey the module identity and give the required module. At a
prescribed rank \(n\), the corresponding existence condition is that
\(-\DD(L)\) lie in the image of \(d_n\). To recover the module
itself from the projective bundle one must also choose an identification
of its inverse lifting twist with \(L\). In classifying-space language
this is the homotopy in a homotopy lift. Changing that identification
can tensor the module by a line bundle.

\begin{theorem}\label{t:primary}
Fix \(k\ge2\) and let \(X\) be a finite connected CW complex. With
amplification \(u\mapsto u\otimes1_k\), form
\(T_k=\operatorname{Tel}_{m\ge0} B\PU(k^m)\).
The image of its lifting class in degree three is
\[
 \operatorname{im}\bigl([X,T_k]\to H^3(X;\Z)\bigr)
  =H^3(X;\Z)[k^\infty]
  :=\bigcup_{m\ge0}\ker(k^m:H^3(X;\Z)\to H^3(X;\Z)).
\]
\end{theorem}
\begin{proof}
Amplification preserves the scalar defect of projective lifts.
The universal lifting maps are therefore compatible up to homotopy
and define a map on the telescope by choosing successive homotopies.
A map from compact \(X\) has bounded telescope coordinate and factors
up to homotopy through a finite stage. The image is the union of the
images of the \(d_{k^m}\), independently of those homotopies.
The determinant bound gives one inclusion. Conversely let
\(k^r\alpha=0\). The Brauer theorem gives an Azumaya representative.
By Antieau--Williams \cite[Thm.~6]{AWprime} it admits a representative
of degree \(d\) with all prime divisors dividing the period of
\(\alpha\). Thus \(d\mid k^m\) for some \(m\). Tensoring with the
trivial matrix algebra of degree \(k^m/d\) realizes \(\alpha\) at that
stage. The zero class has a degree-one representative.
\end{proof}

\begin{example}[The second stage already matters]
Let \(n=k^2\). Clutch the identity of \(\PU(n)\) over
\(X=\Sigma\PU(n)\). The lifting class is, up to sign, the suspension
of \(c_1(\U(n)\to\PU(n))\). In the circle-bundle Gysin sequence
the kernel of \(\Z=H^0(\PU(n);\Z)\xrightarrow{\smile c_1}
H^2(\PU(n);\Z)\) is the image of fiber integration from
\(H^1(\U(n);\Z)\). The determinant generates this group, and its
restriction to the central circle is \(z\mapsto z^n\). The image is
\(n\Z\), so the class has exact order \(k^2\).
\end{example}

The \emph{period} is the order of a Brauer class; the topological
\emph{index} is the greatest common divisor of the degrees representing
it. In particular the period alone cannot specify the least possible
matrix degree. Antieau and Williams exhibit a finite six-dimensional
complex with a period-two class of index eight \cite{AWsix}.
Theorem~\ref{t:primary} permits degrees to grow, an essential
distinction from representability at degree \(k\).

\section{Morita 2-bundle gerbes}\label{sec:higher}

Here the matrix bundles occur on double overlaps. The invariant to
be classified is the component of a descent object, with equivalence
generated by changes of Morita presentation and refinement. We give the
comparisons explicitly because a scalar automorphism of a bimodule
alone does not determine the intertwiner between two different
composites. The geometric theory of ordinary bundle 2-gerbes is due
to Stevenson \cite{Stevenson}.

\subsection{Objects and comparison data}
On each \(U_{ij}\) take a matrix algebra bundle \(A_{ij}\) and on
\(U_{ijk}\) a Morita equivalence
\[
 M_{ijk}: A_{ij}\otimes A_{jk}\longrightarrow A_{ik}.
\]
These are arrows of the bimodule bicategory. Write \(1_A\) for its
regular bimodule. On \(U_{ijkl}\) specify an isomorphism of bimodules
\begin{equation}\label{e:higher-assoc}
 a_{ijkl}:M_{ikl}(M_{ijk}\otimes1_{A_{kl}})
      \Longrightarrow M_{ijl}(1_{A_{ij}}\otimes M_{jkl}).
\end{equation}
Both sides go from \(A_{ij}\otimes A_{jk}\otimes A_{kl}\) to
\(A_{il}\). On \(U_{ijklm}\) require the pentagon: reassociating
the four successive edge factors from the fully left-nested to the
fully right-nested contraction gives the same map around the two
sides of the pentagon. Each step is the indicated \(a\) whiskered
by the identity on the remaining factor or multiplication. The
parenthesizations of ordinary tensor products and their canonical
isomorphisms are understood. In particular \(a\) compares two
\emph{multiplication bimodules}. A diagonal label can be Morita
equivalent to \(\C\) while having any positive matrix degree.

A comparison \((B,N,\varphi):(A,M,a)\to(A',M',a')\) consists of
matrix bundles \(B_i\), Morita equivalences
\[
 N_{ij}:A_{ij}\otimes B_j\longrightarrow B_i\otimes A'_{ij},
\]
and isomorphisms of \((B_i\otimes A'_{ik},\,A_{ij}\otimes A_{jk}
\otimes B_k)\)-bimodules
\begin{align}
 (1_{B_i}\otimes M'_{ijk})(N_{ij}\otimes1_{A'_{jk}})
            (1_{A_{ij}}\otimes N_{jk})
 &\ \xRightarrow{\ \varphi_{ijk}\ }\
 N_{ik}(M_{ijk}\otimes1_{B_k}).\label{e:prism}
\end{align}
On fourfold overlaps require the two pastings to agree: first
reassociate the primed multiplications by \(a'\) and then use
\(\varphi_{jkl},\varphi_{ijl}\), or first use
\(\varphi_{ijk},\varphi_{ikl}\) and then reassociate by \(a\).
All maps are whiskered by the untouched edge arrows. The equality
is written in local coordinates in \eqref{e:prism-cochain} below.
We identify objects linked by such comparisons, allowing zigzags
and common good refinements. Identity comparisons use regular
bimodules. Consecutive comparisons compose by placing their \(B_i\)'s
in tensor product and applying their \(N_{ij}\)'s successively;
the \(\varphi\)'s paste in the same order.

For completeness, a change of comparison data uses Morita-invertible
\(P_i:B_i\to\widetilde B_i\) and intertwiners
\[
 \chi_{ij}:\widetilde N_{ij}(1_{A_{ij}}\otimes P_j)
       \Longrightarrow (P_i\otimes1_{A'_{ij}})N_{ij}.
\]
On triples, applying \(\widetilde\varphi\) and then \(\chi_{ik}\)
agrees with applying \(\chi_{jk},\chi_{ij}\) and then \(\varphi\).
Isomorphisms of these changes are bimodule isomorphisms
\(P_i\to\widehat P_i\) commuting with \(\chi\). Requiring \(P_i\)
to be invertible excludes a rank-two scalar bimodule even when the
base is a point. We shall use these last two levels to describe
trivializations. The classification theorem below concerns objects
up to comparison.

\subsection{Reducing the associator}
Choose splittings \(S_{ij}:A_{ij}\to\C\) with inverse evaluations.
Contracting the matrix actions gives the triple-overlap line
\begin{equation}\label{e:triple}
 \ell_{ijk}
   =S_{ik}M_{ijk}(S_{ij}^{-1}\otimes S_{jk}^{-1}).
\end{equation}
Whisker \eqref{e:higher-assoc} by \(S_{il}\) at the target and by
the three inverse splittings at the source. Inserting their evaluations
identifies the resulting map with
\[
 \ell_{ikl}\otimes\ell_{ijk}\longrightarrow
 \ell_{ijl}\otimes\ell_{jkl}.
\]
All these contractions commute with restriction to a finer cover.
Two contractions on disjoint factors commute; the two contractions
involving a splitting and its inverse agree by the triangle identities.
Thus the transported arrows satisfy the same pentagon. In fact
the original object has a comparison to this scalar one: take
\(B_i=\C\), \(N_{ij}=S_{ij}\); its \(\varphi\) is the evaluation
\(\ell_{ijk}(S_{ij}\otimes S_{jk})\cong S_{ik}M_{ijk}\).

Choose a frame \(e_{ijk}\) for every \(\ell_{ijk}\). Write
\[
 a(e_{ikl}\otimes e_{ijk})=h_{ijkl}\,e_{ijl}\otimes e_{jkl}.
\]
The two paths in the pentagon give
\begin{equation}\label{e:pent}
 h_{ijkl}h_{ijlm}h_{jklm}=h_{iklm}h_{ijkm},
 \qquad\text{that is, }\delta h=1.
\end{equation}
Replacing \(e_{ijk}\) by \(c_{ijk}e_{ijk}\) replaces \(h\) by
\(h/(\delta c)\). To change the splittings put
\(D_{ij}=\widehat S_{ij}S_{ij}^{-1}\). This is a line on the
double overlap and
\[
 \widehat\ell_{ijk}\cong
 D_{ik}\otimes\ell_{ijk}\otimes D_{ij}^\vee\otimes D_{jk}^\vee .
\]
Choose frames of the \(D_{ij}\) and their dual frames. In these
induced triple frames the \(D\)-factors cancel on both sides of the
associator: only the outside factors \(D_{il},D_{ij}^\vee,
D_{jk}^\vee,D_{kl}^\vee\) remain and the map on them is the
identity. Thus the two choices give the same \(h\) in matched
frames, and cohomologous \(h\)'s in arbitrary frames.

\begin{theorem}\label{t:four}
For the good-cover bases specified in the introduction, tensor product
and pullback give a natural isomorphism
\[
 \{\text{Morita 2-bundle gerbes on }X\}/\text{equivalence}
       \ \cong\ \check H^3(X;\CX)\ \cong\ H^4(X;\Z).
\]
\end{theorem}
\begin{proof}
It remains to check what an arbitrary comparison does to \(h\).
Choose splittings \(T_i:B_i\to\C\) and splittings of both systems of
edge algebras. The transported \(N_{ij}\) is the line
\[
 n_{ij}=(T_i\otimes S'_{ij})\,N_{ij}\,
                  (S_{ij}^{-1}\otimes T_j^{-1}).
\]
Cancelling adjacent splittings turns \(\varphi_{ijk}\) into
\[
 \ell'_{ijk}\otimes n_{ij}\otimes n_{jk}\longrightarrow
             n_{ik}\otimes\ell_{ijk}.
\]
Take a frame of each \(n_{ij}\); let \(b_{ijk}\) be the scalar
of this map from its indicated domain frames to its codomain frames.
The two fourfold pastings in the definition now read
\begin{equation}\label{e:prism-cochain}
 h'_{ijkl}\,b_{jkl}b_{ijl}=b_{ijk}b_{ikl}\,h_{ijkl},
 \qquad h'=h/(\delta b).
\end{equation}
These coordinates detect equality of intertwiners: conjugation by
splittings is an equivalence of the corresponding bimodule categories,
and change of frame is an isomorphism. In particular matrix labels on
the patches of a comparison cannot identify two distinct cohomology
classes. A change of frame of \(n_{ij}\) by \(d_{ij}\) changes \(b\)
by \(\delta d\). Refinement restricts all the cochains.

Every scalar 3-cocycle occurs with trivial edge algebras, trivial
triple lines and associator \(h\) by \eqref{e:pent}. If two cocycles
differ by \(h'=h/(\delta b)\), take \(B_i=\C\), trivial \(n_{ij}\)
and \(\varphi_{ijk}\) given by \(b_{ijk}\). Equation
\eqref{e:prism-cochain} is exactly the required compatibility.
Every matrix object is linked to its scalar reduction by the comparison
following \eqref{e:triple}. This proves injectivity as well as
surjectivity. Tensor products multiply \(h\), and pullback pulls it
back as a cochain, proving the stated natural group identification.
\end{proof}

\paragraph{Trivializations and the scope of the comparison.}
In the scalar model a trivialization of \(h\) amounts to \(b\) with
\(\delta b=h\). With fixed frames, a change of comparison has
\(\chi_{ij}\) given by a scalar \(q_{ij}\). The triple condition is
\(\widetilde b_{ijk}q_{ik}=q_{ij}q_{jk}b_{ijk}\), or
\(\widetilde b/b=\delta q\). Isomorphisms of the \(P_i\) change
\(q\) by vertex scalars. The classes of scalar trivializations,
when nonempty, therefore form a torsor for
\(\check H^2(X;\CX)=H^3(X;\Z)\). For a comparison from the
trivial matrix object to itself, the data are precisely \(B_i\),
invertible \(N_{ij}:B_j\to B_i\), and associative contractions
\(N_{ij}N_{jk}\to N_{ik}\). They are ordinary Morita gerbe data
with the overlap arrows indexed in reverse order; changes of
comparison are their ordinary equivalences. This recovers \(H^3(X;\Z)\) as the group of self-equivalences of the unit modulo 2-equivalence.

The suspension construction also has a geometric description here.
Cover \(S^4=\Sigma S^3\) by two open cones with equatorial overlap
homotopy equivalent to \(S^3\). Take the trivial 2-gerbe on each
cone. On the overlap glue by an ordinary gerbe with
\(\DD\) a generator of \(H^3(S^3;\Z)\). To see its class, pass
to a common good refinement. The trivializations on the two cones
are represented by local scalar 2-cochains; on their overlap their
quotient represents the gerbe used for gluing. The
Mayer--Vietoris boundary therefore gives its class in
\(H^4(S^4;\Z)\), a generator up to sign. Equivalently this is the
usual construction of the boundary by extending a cocycle to
cochains on the two sides and taking their coboundaries.

Theorem~\ref{t:four} proves a statement about equivalence classes. A
triequivalence of full higher models would additionally require
coherent comparison functors on all hom-bicategories. The degree-four
class has its usual 2-gerbe interpretation. For example \(S^4\)
has \(H^4(S^4;\Z)=\Z\), whereas
\(\pi_4 B\GL_1(\KU)=\pi_3\KU=0\). A realization by modules in a
higher twisted cohomology theory would require a specified target
and action.

\section{A matrix bundle with two different stabilization invariants}
\label{sec:uhf}

For a matrix bundle the finite Brauer class forgets the choice of
simple module. Stabilization to a fixed \emph{unital} UHF fiber can
retain some of that information. This distinction is relevant to
the description of bundles by local diagrams of matrix embeddings
(LABs). On a global \(M_k\)-bundle the unital construction is just
extension of structure group along
\[
 \PU(k)\longrightarrow\Aut(D),\qquad [u]\longmapsto
          \operatorname{Ad}(u)\otimes\id_A,\qquad
 A=M_{l^\infty},\quad D=M_k\otimes A .
\]
Here \(\Aut(D)\) has the point-norm topology. The construction below
uses only global bundles and does not require a classification theorem
for local diagrams.

\begin{proposition}\label{p:sphere}
Let \(k,l\ge2\) be coprime, let \(J\) be a Hopf line on \(S^2\) with
\(c_1(J)\) a generator, and put \(V_j=J^{\otimes j}\oplus\one^{k-1}\)
for \(j\in\Z\). Each \(\End(V_j)\) has zero Brauer class. The loop
clutching \(\End(V_1)\otimes A\) has order \(k\) in \(\pi_1\Aut(D)\).
In particular \(\End(V_j)\otimes A\) is a trivial unital \(D\)-bundle
exactly when \(k\mid j\).
\end{proposition}
\begin{proof}
The projective frame bundle lifts to the unitary frame bundle of \(V_j\),
so its lifting class is zero. The determinant of a clutching loop for
\(V_j\) has winding number \(j\) after fixing the orientation.
The exact sequence for the central circle shows that the projected
loop represents \(j\bmod k\) in \(\pi_1\PU(k)\). In particular
the amplified loop is null when \(k\mid j\).

For the converse, set \(B=C(S^2)\otimes A\) and let
\(W_j=\Gamma(V_j)\otimes A\) be the corresponding full finite
projective right \(B\)-module. Then the section algebra of the
amplified bundle is \(\End_B(W_j)\). Its unit corresponds under the
Morita isomorphism \(K_0(\End_B(W_j))\cong K_0(B)\) to \([W_j]\).
Put \(t=[J]-1\). The identities
\[
 K^0(S^2)=\Z[t]/(t^2),\qquad
 K_0(B)=\Z[1/l]\oplus\Z[1/l]\,t
\]
give \([W_j]=k+jt\). The second identity follows by continuity of
\(K_0\) for \(C(S^2)\otimes M_{l^n}\); amplification acts by \(l\)
in matrix Morita coordinates, and the normalization sends \([1_A]\)
to \(1\). In the section algebra of the trivial bundle the unit
is \(k\)-divisible in \(K_0\), since its \(k\) diagonal matrix
projections are equivalent. Divisibility by \(k\) is invariant under
a unital isomorphism. If our bundle is trivial, \(k+jt\) must be
\(k\)-divisible, so \(j/k\in\Z[1/l]\). Coprimality gives \(k\mid j\).
Since these clutching loops are the images of the multiples of a single
generator of \(\pi_1\PU(k)\), the asserted order follows.
\end{proof}

For this UHF type Thomsen's calculation gives
\(\pi_1\Aut(D)\cong\Z/k\) \cite{Thomsen}, so the loop for \(j=1\)
is a generator. There is also a finite-stage description. Evaluation
on the standard \(M_k\) takes the amplified loop to
\[
 \operatorname{Emb}(M_k,M_{kl^N})
       =\PU(kl^N)/(1_k\otimes\PU(l^N)).
\]
Here one first amplifies into \(M_{kl^N}\) and then restricts to
\(M_k\). On fundamental groups the stabilizer contributes the subgroup
generated by \(k\) in \(\Z/(kl^N)\). The original clutching loop maps
to \(l^N\bmod k\) in the quotient. It remains a generator at every
stage. The unit-divisibility argument above detects it directly in
the completed \(D\)-bundle. Thus a global matrix bundle that is
Morita-trivial can still carry nontrivial information in the
matrix-embedding theory.

This example also locates the comparison with the operator-algebraic
Brauer theory of Dadarlat and Pennig \cite{DPBrauer}, which uses
stabilized fibers \(A\otimes\mathcal K\) and a tensor product of
classes. The fixed unital algebra \(M_k\otimes A\) in
Proposition~\ref{p:sphere} carries a distinguished unit. Its class in
ordered \(K_0\) is the ingredient detected by the example.


\begin{thebibliography}{99}

\bibitem{ABGHR} M. Ando, A. J. Blumberg, D. Gepner, M. J. Hopkins and
C. Rezk, \emph{Units of ring spectra and Thom spectra}. \arxiv{0810.4535}.

\bibitem{AWsix} B. Antieau and B. Williams, \emph{The topological
period--index problem over 6-complexes}. \arxiv{1208.4430}.

\bibitem{AWprime} B. Antieau and B. Williams, \emph{The prime divisors
of the period and index of a Brauer class}. \arxiv{1403.3770}.

\bibitem{BCMMS} P. Bouwknegt, A. L. Carey, V. Mathai, M. K. Murray
and D. Stevenson, \emph{Twisted \(K\)-theory and \(K\)-theory of bundle
gerbes}, Comm. Math. Phys. \textbf{228} (2002), 17--45.
\arxiv{hep-th/0106194}.

\bibitem{BSch} U. Bunke and T. Schick, \emph{On the topology of
\(T\)-duality}, Rev. Math. Phys. \textbf{17} (2005), 77--112.

\bibitem{DMP} M. Dadarlat, J. E. McClure and U. Pennig,
\emph{Computing cohomology groups that classify bundles of strongly
self-absorbing \(C^*\)-algebras}, J. Topol. Anal. \textbf{17} (2025),
1591--1624. \arxiv{2302.08028}.

\bibitem{DPBrauer} M. Dadarlat and U. Pennig, \emph{A Dixmier--Douady
theory for strongly self-absorbing \(C^*\)-algebras II: the Brauer
group}, J. Noncommut. Geom. \textbf{9} (2015), 1137--1154.
\href{https://doi.org/10.4171/JNCG/218}{doi:10.4171/JNCG/218}.

\bibitem{original} A. V. Ershov, \emph{Morita bundle gerbes},
2016, revised 2017. \arxiv{1610.05754}.

\bibitem{HS} J. Hornbostel and S. Schr\"oer, \emph{Infinite CW-complexes,
Brauer groups and phantom cohomology}. \arxiv{1304.3986}.

\bibitem{Karoubi2} M. Karoubi, \emph{Twisted bundles and twisted
\(K\)-theory}, Clay Math. Proc. \textbf{16} (2012). \arxiv{1012.2512}.

\bibitem{KLW} P. Kristel, M. Ludewig and K. Waldorf,
\emph{2-vector bundles}, Higher Structures \textbf{9} (2025), 36--87.
\href{https://doi.org/10.21136/HS.2025.02}{doi:10.21136/HS.2025.02}.

\bibitem{Kuiper} N. H. Kuiper, \emph{The homotopy type of the unitary
group of Hilbert space}, Topology \textbf{3} (1965), 19--30.

\bibitem{MayS} J. P. May and J. Sigurdsson, \emph{Parametrized
homotopy theory}, Mathematical Surveys and Monographs \textbf{132},
Amer. Math. Soc., 2006.

\bibitem{Murray} M. K. Murray, \emph{Bundle gerbes}, J. London
Math. Soc. (2) \textbf{54} (1996), 403--416.

\bibitem{MS} M. K. Murray and D. Stevenson, \emph{Bundle gerbes:
stable isomorphism and local theory}, J. London Math. Soc. (2)
\textbf{62} (2000), 925--937. \arxiv{math/9908135}.

\bibitem{Pennig} U. Pennig, \emph{Twisted \(K\)-theory with
coefficients in \(C^*\)-algebras}. \arxiv{1103.4096}.

\bibitem{Stevenson} D. Stevenson, \emph{Bundle 2-gerbes},
Proc. London Math. Soc. (3) \textbf{88} (2004), 405--435.
\arxiv{math/0106018}.

\bibitem{Thomsen} K. Thomsen, \emph{The homotopy type of the group
of automorphisms of a UHF-algebra}, J. Funct. Anal. \textbf{72}
(1987), 182--207.
\href{https://doi.org/10.1016/0022-1236(87)90085-1}{doi:10.1016/0022-1236(87)90085-1}.

\end{thebibliography}
\end{document}